\documentclass[10pt]{amsart}

\usepackage{enumerate,amsrefs,nicefrac,graphicx,amssymb}

\newtheorem{theorem}{Theorem}[section]

\theoremstyle{definition}
\newtheorem{definition}[theorem]{Definition}

\newtheorem{corollary}[theorem]{Corollary}

\newtheorem{fact}[theorem]{Fact}
\newtheorem{question}[theorem]{Question}

\theoremstyle{remark}
\newtheorem{remark}[theorem]{Remark}

\numberwithin{equation}{section}

\begin{document}

\title{Reservoir Computing Dynamics for Single Nonlinear Node with Delay Line Structure}
\author{Claudio A. DiMarco}
\date{\today}

\maketitle

\begin{abstract}
For a reservoir computer composed of a single nonlinear node and delay line, we show that after a finite period of discrete time, the distance between two reservoir outputs is bounded above by a constant multiple of the distance between their respective inputs.  We also translate familiar separation properties from the context of Echo State Networks to that of the single nonlinear node structure.
\end{abstract}

\section{Introduction}
Many recent studies have been done on the construction, behavior, and performance of reservoir computers.  Two popular reservoir structures are the echo state network (ESN) and the single nonlinear node with delay line~\cite{appeltantetal}.  Jaeger provided a mathematical description and analysis of echo states, in which case the reservoir is assumed to have a finite number of randomly interconnected internal nodes~\cite{jaeger}.  While much work has been done to construct, test, and analyze single-node reservoir schemes~\cites{appeltant,appeltantetal,paquot,duport,brunner}, there remains a lack of understanding and an absence of a rigorous mathematical analysis of the underlying dynamics of such a system.

We seek to provide some mathematical insight into the dynamics of the single-node reservoir computer.  This analysis is done in two parts.  Section \ref{results1} describes how the variation in outputs (in the sense of the $\ell^2$ norm) is bounded above by a constant multiple of the variation of time-series input vectors.  To this end, we draw on basic mathematical tools to prove something of a Lipschitz continuity condition on the entire system:
\newtheorem*{main_result}{Theorem \ref{main_result}}
\begin{main_result}
Suppose $f:\mathbb{R}\rightarrow \mathbb{R}$ is $L$-Lipschitz with $\alpha L < \nicefrac{1}{\sqrt{2}},$ and define $y_u:\mathbb{Z}_{\geq 0} \rightarrow \mathbb{R}$ by $y_u(t)=\sum_{k=0}^{N} w_k x_k^{(u)}(t)$ for an input spike train $u\in\mathbb{R}^M.$  Then
\begin{center}
$\| y_u-y_v \|_{\ell^2(\mathbb{R}^M)}^2 \leq |w|^2 M(L\beta)^2 (N+1) \left(1+\frac{2}{1-2\alpha^2 L^2} \right) \| u-v \|_{\ell^2(\mathbb{R}^M)}^2$
\end{center}
for all inputs $u,v.$
\end{main_result}

Section \ref{classification_section} provides a translation of some well known classification metrics~\cites{goodman,gibbons} to the context of the single-node with delay line reservoir structure.  That section also contains a brief review and analysis of the importance of selecting an injective nonlinear sigmoid function for the single  nonlinear input node.  In particular, we show that classification mishaps are possible if the injectivity condition is disregarded, or if the input data are not properly pre-conditioned.

\section{Preliminaries}\label{preliminaries}

\subsection{Notation and definitions}\label{notation_subsection}

Let $t\in \mathbb{Z}_{\geq 0}$ denote discrete time and let $\alpha, \beta \in (0,1)$.  An \textit{input} is a vector $u \in \mathbb{R}^m$ where the $0$th entry is defined as $u(0)=0.$  Since we only consider reservoirs with a single nonlinear node and delay line, the term \textit{reservoir} will be reserved for that structure.
\begin{definition}\label{resdef}
Let $f:\mathbb{R}\rightarrow\mathbb{R}$ be non-linear.  A \textit{reservoir of length} $N$, \textit{nonlinearity} $f$, \textit{input gain} $\beta,$ \textit{and feedback gain} $\alpha,$ denoted $X(N,f,\beta,\alpha),$ is a set of $N$ functions $X=\{x_1,x_2,\dots,x_N\},$ where $ x_k:\mathbb{Z}_{\geq 0} \rightarrow \mathbb{R},$ defined by the relations \begin{equation}\label{resdefeqn}
\begin{split}
x_k(0) &= 0~\text{for all}~k\geq 0 \\
x_1(t) &= f(\alpha x_N(t-1) + \beta u(t))~\text{for}~t\geq 1\\
x_{k+1}(t) &= x_k(t-1)~\text{for all}~k~\text{and for}~t\geq 1
\end{split}
\end{equation}
We will have occasion to use a modified version of Definition \ref{resdef} in which $x_1(t)= f(\alpha x_N(t) + \beta u(t))$ instead of $x_1(t)= f(\alpha x_N(t-1) + \beta u(t))$ in \eqref{resdefeqn}.
\end{definition}
Call $x(t)=\{x_1(t),x_2(t),\dots,x_N(t)\}$ the \textit{state of the reservoir at time} $t.$  A vector/set of \textit{weights} is $(w_1,w_2,\dots,w_n)\in \mathbb{R}^N,$ where $w_k$ is the weight of the ``node" $x_k.$  Given a set of weights, the \textit{output of the reservoir at time $t$} is defined as the sum of weighted node states $y(t) = \sum_{k=1}^N w_k x_k(t).$  Suppose we are working with a given data set of (masked) inputs, call it $\mathcal{\tilde{U}}$.  Put $M=\max\{m:u\in\mathbb{R}^m\}$. The respective reservoir outputs will be considered over the interval of discrete time $[1,M].$  In order to compare two inputs $u$ and $v$ it is convenient to simply extend all input vectors of dimension less than $M$.  To this end, if $u\in\mathbb{R}^m$ and if $m<M,$ identify $u$ with $(u(1),u(2),\dots,u(m),0,\dots,0)\in\mathbb{R}^M.$  From this point forward we will write $\mathcal{U}$ as the set of all inputs after this elongation process is performed.  Writing $y_u(t)$ for the reservoir output at time $t$ with respect to an input $u,$ a discrete time interval $[1,M]$ yields an output vector $y_u \in \mathbb{R}^{M}.$  So it is reasonable to consider the $\ell^2$ norm of $u$ or $y_u,$ denoted $\|u\|_{\ell^2(\mathbb{R}^M)}$, or just $\|u\|_2$ where the dimension of $u$ is clear from the context.

\begin{center}
	\begin{figure}[h]
		\includegraphics[width=0.99\textwidth]{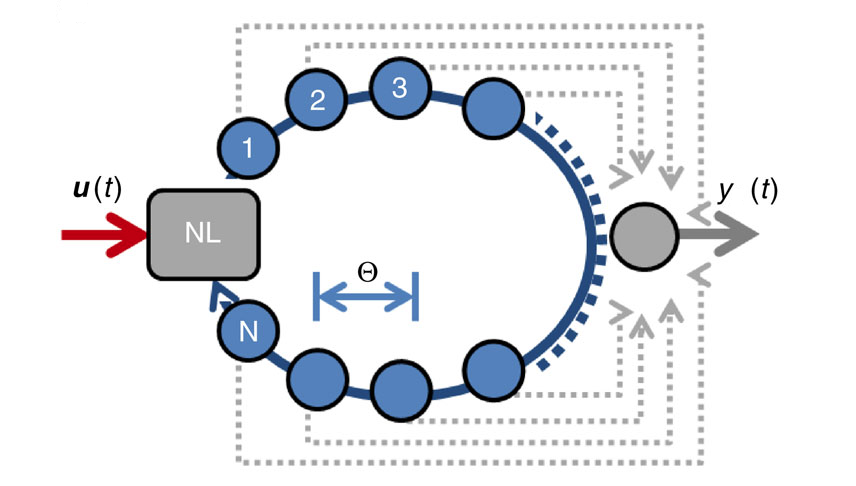}\caption{(Brunner et al., 2013)}\label{caption}
	\end{figure}
\end{center}

\begin{center}
    \begin{tabular}{ | l | l | l | p{5cm} |}
    \hline
    Notation & Meaning \\ \hline \hline
    $u(t)$ & coordinate $t$ of input $u$ \\ \hline
    $x_n^{(u)}(t)$ & state of $n$th node at time $t$ with input $u$ \\ \hline
    $\alpha$ & feedback gain \\ \hline
		$\beta$ & input gain \\ \hline
		$w_i$ & weight of $i$th node state \\ \hline
    \end{tabular}
\end{center}

\subsection{Weight training}
When preparing a reservoir to perform a task, it is necessary to choose a set of weights.  We assume the existence of a data set consisting of a collection of inputs.  One way or another, these data are used to determine (train) appropriate weights.  To train a set of weights, one needs a set of inputs and a corresponding set of target outputs $\{\hat{y}(t)\}$ to be compared to the actual machine outputs.  That is, one finds $\{w_k\}$ such that the variation between $y(t)=\sum_{k=1}^N w_k x_k(t)$ and $\hat{y}(t)$ is minimized.

There is a variety of methods to train weights that is used in the literature including least squares, ridge regression, minimization of mean square error, and minimization of normalized mean square error~\cites{appeltant,paquot,duport}.  Ridge regression is commonly used to accomodate systems with several parameters to avoid overfitting weights to training data.  The goal is to ensure that the system does not categorize two distinct members of the same input class as ``different."    One of the most popular benchmark tasks for testing a reservoir computer is NARMA, wherein the Normalized Root Mean Square Error (NRMSE) is minimized~\cite{appeltant}.

The authors of \cite{prater} developed an algorithm for computing what are called Dantzig selectors, which were defined by Candes and Tao in \cite{candes}.  Let $\delta >0,~ X\in M_{n\times p}(\mathbb{R}),$ and $y\in \mathbb{R}^n.$  Let $D$ be a diagonal matrix where $D(k,k)$ is the $\ell^2$ norm of the $k$th column of $X.$ The Dantzig selector, denoted $\hat{\beta}\in\mathbb{R}^p$, solves the following optimization problem:
\begin{equation*}
\hat{\beta}\in \text{argmin}\{\|\beta\|_1 : D^{-1} X^T(X\beta - y)\|_{\infty} \leq \delta\}.
\end{equation*}
This concept can be used to train reservoir node weights by storing the states of the reservoir over time as the columns of $X$; that is, $X(:,t)=X(t)$.


\section{An upper bound on input $\mapsto$ output error}\label{results1}

To ensure that an input $\mapsto$ output system is reasonably accurate, it is desirable to have a growth condition that places an upper bound on the distortion of distances between two outputs with respect to the distance between their corresponding inputs.  The following is a classical definition that describes such a condition.

\begin{definition}\label{Lipschitz_definition}
Suppose $X$ and $Y$ are metric spaces and $L\geq 0.$  A mapping $f:X\rightarrow Y$ is \textit{$L$-Lipschitz} if
\begin{equation*}
d_Y(f(z),f(w)) \leq L d_X(z,w)
\end{equation*}
for all $z,w\in X.$  The function $f$ is \textit{Lipschitz} if it is $L$-Lipschitz for some $L\geq 0.$
\end{definition}
\noindent If $f:\mathbb{R}\rightarrow\mathbb{R}$ is differentiable, then $f'$ is bounded if and only if $f$ is Lipschitz.  In fact, if $f:\mathbb{R}\rightarrow \mathbb{R}$ is any Lipschitz function, then $f$ is absolutely continuous, and is differentiable except on a set $E\subset \mathbb{R}$ of Lebesgue measure $0.$  Definition \ref{Lipschitz_definition} says that $f$ can only increase distances by a bounded amount.

In this section we assume that there is a specific task at hand, and that target outputs have been defined and used to train a set of weights.  It is also assumed that for the set of inputs $\mathcal{U}$, each $u\in\mathcal{U}$ has been standardized to have dimension $M$ as described in Subsection \ref{notation_subsection}.  In other words, we only consider inputs over the discrete time interval $[1,M]\cap \mathbb{N}$.

For the purpose of classification in a reservoir system, it is desirable that the distance between two outputs is controlled by the distance between their respective inputs.  The Lipschitz condition clearly describes this property.  The main result in this section is that the function $u\mapsto y_u$ is Lipschitz, where the Lipschitz constant depends on $N,\alpha, \beta,$ and $M.$  This result does not appear to be optimal, however, since the Lipschitz constant is directly proportional to the dimension of the input vectors and the number of nodes in the reservoir.  Recall Definition \ref{resdef}:

\begin{definition}\label{1} For a system with $N+1$ nodes and input $u,$
\begin{equation}
\begin{split}
x_0(t) &= f( \alpha x_N(t-1) + \beta u(t)) \\
x_{n+1}(t) &= x_n(t-1),~~n=0,\dots,N-1 \\
x_n(0) &= 0\\
x_0(1) &= f(0)\\
\end{split}
\end{equation}
\end{definition}

 The following couple of modest observations will be useful for proving Theorem \ref{main_result}.

\begin{fact}\label{2}
If $0\leq t \leq N$ then $x_N(t)=0$ and $x_N(N+t)=f(\beta u(t))$.
\end{fact}

\begin{corollary}\label{3}
If $t\in [1,N]\cap \mathbb{Z}$ then $x_0(t) = f(\beta u(t))$.
\end{corollary}
\begin{proof}
Since $t-1\leq N,$ Definition \ref{1} and Fact \ref{2} yield
\begin{equation*}
\begin{split}
x_0(t) &= f(\alpha x_N(t-1) + \beta u(t))\\
&= f(\alpha x_{N-t+1}(0) + \beta u(t))\\
&= f(\beta u(t)). \qedhere
\end{split}
\end{equation*}
\end{proof}

It would be convenient to have a Lipschitz continuity condition on the function $u(t)\mapsto y_u(t)$ where $t$ is fixed.  This idea seems naive, however, since we do not expect the difference of reservoir outputs $|y_u(t)-y_v(t)|$ to depend entirely on the difference in one entry of the inputs $|u(t)-v(t)|.$  That is, we do not know how to capture the behavior of the reservoir simply by looking at the input difference at a particular time step.  An approach to predicting classification accuracy involves the concept of ``separation" of inputs and reservoir states, which will be discussed in Section \ref{classification_section}.

\begin{theorem}\label{main_result}
Suppose $X(N,f,\beta,\alpha)$ is a reservoir and $f:\mathbb{R}\rightarrow \mathbb{R}$ is $L$-Lipschitz with $\alpha L < \nicefrac{1}{\sqrt{2}}.$  Then for two inputs $u$ and $v,$
\begin{center}
$\| y_u-y_v \|_{\ell^2(\mathbb{R}^M)}^2 \leq |w|^2 M(L\beta)^2 (N+1) \left(1+\frac{2}{1-2\alpha^2 L^2} \right) \| u-v \|_{\ell^2(\mathbb{R}^M)}^2.$
\end{center}
\end{theorem}
\begin{remark}
The bars $\|\cdot\|$ indicate the length of the vector as a time series, whereas the bars $|\cdot|$ are used for all other lengths.
\end{remark}
\begin{proof}
For all $t,$
\begin{equation}
\begin{split}\label{4}
|y_u(t)-y_v(t)| &= \left|\sum_{i=0}^N w_i (x_i^{(u)}(t) - x_i^{(v)}(t))\right| \\
&\leq \sum_{i=0}^N |w_i| \cdot |x_i^{(u)}(t) - x_i^{(v)}(t)|.
\end{split}
\end{equation}

If $t=0$ there is nothing to prove.  We will deal with special time intervals and achieve the above estimation on each interval.

\begin{enumerate}[(I)]
\item ${1\leq t \leq N+1}$:\\
\begin{center}
$x_0(t) = f(\alpha x_N(t-1) + \beta u(t)) = f(\beta u(t))$ \\
\begin{displaymath}
  x_1(t)=\left\{
  \begin{array}{lr}
f(\beta u(t-1) & ,~ t\geq 2 \\
0 & ,~ t=1
\end{array}
\right.
\end{displaymath}
\\
\begin{displaymath}
x_2(t)=\left\{
\begin{array}{lr}
f(\beta u(t-2) & ,~ t\geq 3 \\
0 & ,~ t=1,2
\end{array}
\right.
\end{displaymath}
\vdots \\
\begin{displaymath}
x_i(t) = \left\{
\begin{array}{lr}
f(\beta u(t-i)) & ,~ t\geq i\\
0 & ,~ t<i
\end{array}
\right.
\end{displaymath}
\end{center}
In this case,
\begin{equation}
\begin{split}\label{5}
\sum_{i=0}^N |w_i|\cdot|x_i^{(u)}(t)-x_i^{(v)}(t)|
&\leq \sum_{i=0}^N |w_i|\cdot |f(\beta u(t-i))-f(\beta v(t-i))| \\
&\leq \sum_{i=0}^N |w_i|\cdot L\beta|u(t-i)-v(t-i)|\\
&\leq L\beta \sum_{i=0}^N |w_i| \cdot |u(t-i)-v(t-i)|.
\end{split}
\end{equation}
It follows from \eqref{4}, \eqref{5}, and the Schwarz inequality that
\begin{equation}
\begin{split}\label{6}
|y_u(t)-y_v(t)|^2 &\leq L^2 \beta^2 \left(\sum_{i=0}^N |w_i| \cdot |u(t-i)-v(t-i)| \right)^2 \\
&\leq L^2 \beta^2 \sum_{i=0}^N |u(t-i)-v(t-i)|^2 \sum_{i=0}^N |w_i|^2 \\
&= |w|^2 L^2 \beta^2 \sum_{i=0}^N |u(t-i)-v(t-i)|^2.
\end{split}
\end{equation}

Write $C=|w|^2 L^2 \beta^2 $.  Note that the function $t\mapsto t-i$ is linear and hence injective.  Therefore equation \eqref{6} gives
\begin{equation*}
\begin{split}
\sum_t |y_u(t)-y_v(t)|^2 &\leq \sum_t C\sum_{i=0}^N |u(t-i)-v(t-i)|^2 \\
&= C\sum_{i=0}^N \sum_t |u(t-i)-v(t-i)|^2 \\
&\leq C\sum_{i=0}^N \|u-v\|_{\ell^2(\mathbb{R}^{N+2})}^2 \\
&= C(N+1) \|u-v\|_{\ell^2(\mathbb{R}^{N+2})}^2,
\end{split}
\end{equation*}
so that
\begin{equation}\label{7}
\|y_u-y_v\|_{\ell^2(\mathbb{R}^{N+2})}^2 \leq |w|^2 (N+1) (L \beta)^2 \|u-v\|_{\ell^2(\mathbb{R}^{N+2})}^2.
\end{equation}

\item $N+1\leq t\leq 2N+2$

Note that for $0\leq k \leq N$ we have
\begin{equation}
\begin{split}
x_0(N+1+k)&=f(\alpha x_N(N+k) + \beta u(N+k+1))\\
&= f(\alpha x_0(k) + \beta u(N+k+1))\\
&= f(\alpha f(\beta u(k)) + \beta u(N+1+k)).
\end{split}
\end{equation}
So the node states can be written as follows:
\begin{center}
$x_0(t)=f(\alpha f(\beta u(t-N-1)) + \beta u(t))$ \\

\begin{displaymath}
x_1(t)= \left\{
\begin{array}{lr}
f(\alpha f(\beta u(t-N-2)) + \beta u(t-1)) & ,~ t\geq N+3\\
f(\beta u(t-1)) & ,~ t=N+1,N+2
\end{array}
\right.
\end{displaymath}
\\

\begin{displaymath}
x_2(t)= \left\{
\begin{array}{lr}
f(\alpha f(\beta u(t-N-3)) + \beta u(t-2)) & ,~ t\geq N+4\\
f(\beta u(t-2)) & ,~ t=N+1,N+2,N+3
\end{array}
\right.
\end{displaymath}
\vdots
\begin{displaymath}
x_i(t)= \left\{
\begin{array}{lr}
f(\alpha f(\beta u(t-N-(i+1)) + \beta u(t-i)) & ,~ t\geq N+i+2\\
f(\beta u(t-i)) & ,~ t=N+1,N+2,\dots,N+i+1
\end{array}
\right.
\end{displaymath}
\end{center}

Fix $t$ and write $t=N+2+i_0$ for some $i_0 \in \{0,1,\dots,N\}.$  Then

\begin{displaymath}
x_i(t) = \left\{
\begin{array}{lr}
f(\alpha f(\beta u(t-N-i-1)) + \beta u(t-i)) & ,~ 0\leq i\leq i_0 \\
f(\beta (t-i)) & ,~ i_0 < i \leq N \\
\end{array}
\right.
\end{displaymath}
So $|y_u(t) - y_v(t)|$ is bounded above by
\begin{equation}
\begin{split}\label{10}
& \sum_{0\leq i\leq i_0} |w_i| \cdot |f(\alpha f(\beta u(t-N-i-1)) + \beta u(t-i)) - [f(\alpha f(\beta v(t-N-i-1)) + \beta v(t-i))]| \\
&\hspace{1cm} + \sum_{i_0 < i \leq N} |w_i| \cdot |f(\beta u(t-i)) - f(\beta v(t-i))|
\end{split}
\end{equation}
Also
\begin{equation}
\begin{split}\label{8}
& |f(\alpha f(\beta u(t-N-i-1)) + \beta u(t-i)) - f(\alpha f(\beta v(t-N-i-1)) + \beta v(t-i))| \\
&\leq L|\alpha f(\beta u(t-N-i-1)) + \beta u(t-i) - [\alpha f(\beta v(t-N-i-1)) + \beta v(t-i)]| \\
&= L|\alpha[f(\beta u(t-N-i-1))-f(\beta v(t-N-i-1))] + \beta[u(t-i)-v(t-i)]| \\
&\leq L\left[ \alpha|f(\beta u(t-N-i-1))-f(\beta v(t-N-i-1))| + \beta|u(t-i)-v(t-i)| \right] \\
&\leq L[ \alpha L \beta |u(t-N-i-1)-v(t-N-i-1)| + \beta|u(t-i)-v(t-i)| ] \\
&\leq \alpha L^2 \beta |u(t-N-i-1)-v(t-N-i-1)| + L\beta|u(t-i)-v(t-i)|.
\end{split}
\end{equation}
Thus the square of the quantity \eqref{10} is bounded above by
\begin{equation}
\begin{split}\label{9a}
& [~~\alpha L^2 \beta \sum_{0\leq i\leq i_0} |w_i| \cdot |u(t-N-i-1)-v(t-N-i-1)| + L\beta \sum_{i_0 < i \leq N} |w_i| \cdot |u(t-i)-v(t-i)| ~~]^2 \\
&\leq 2(L\beta)^2 [~~ (\alpha L)^2 \left(\sum_{0\leq i\leq i_0} |w_i| \cdot |u(t-N-i-1)-v(t-N-i-1)| \right)^2 \\
&\hspace{1cm} + \left(\sum_{i_0 < i \leq N} |w_i| \cdot |u(t-i)-v(t-i)|\right)^2 ] \\
&\leq 2(L\beta)^2 [~~ (\alpha L)^2 \sum_{0 \leq i \leq i_0} |w_i|^2 |u(t-N-i-1)-v(t-N-i-1)|^2 \\
&\hspace{1cm} + \sum_{i_0 < i \leq N} |w_i|^2 \cdot |u(t-i)-v(t-i)|^2 ~~ ]
\end{split}
\end{equation}
It now follows that
\begin{equation}
\begin{split}
\|y_u - y_v\|_{\ell^2(\mathbb{R}^{N+2})}^2 &\leq (L\beta)^2 (N+1) (2(\alpha L)^2 +2)|w|^2 \cdot \|u-v\|_{\ell^2(\mathbb{R}^M)}^2
\end{split}
\end{equation}
Note that we have used the Schwarz inequality and the convexity of the function $x\mapsto x^2$ to obtain \eqref{9a}.

\item $2N+2 \leq t \leq 3N+3$ \\

\begin{center}
$x_0(t) = f(\alpha f(\alpha f(\beta u(t-2N-2)) + \beta u(t-N-1)))$ \\

\begin{displaymath}
x_1(t) = \left\{
\begin{array}{lr}
f(\alpha f(\alpha f(\beta u(t-2N-3)) + \beta u(t-N-2))) & ,~ t\geq 2N+3 \\
f(\alpha f(\beta u(t-N-2)) + \beta u(t-1)) & ,~ t = 2N+2
\end{array}
\right.
\end{displaymath}
\\

\begin{displaymath}
x_2(t) = \left\{
\begin{array}{lr}
f(\alpha f(\alpha f(\beta u(t-2N-4)) + \beta u(t-N-3))) & ,~ t \geq 2N+4 \\
f(\alpha f(\beta u(t-N-3)) + \beta u(t-2)) & ,~ t = 2N+2, 2N+3
\end{array}
\right.
\end{displaymath}
\vdots \\
\begin{displaymath}
x_i(t) = \left\{
\begin{array}{lr}
f(\alpha f(\alpha f(\beta u(t-2N-i-2)) + \beta u(t-N-i-1))) & ,~ t \geq 2N+i+2 \\
f(\alpha f(\beta u(t-N-i-1)) + \beta u(t-i)) & ,~ 2N+2 \leq t \leq 2N+i+1
\end{array}
\right.
\end{displaymath}
\end{center}
Let $t = 2N+3+ i_0,~ i_0 \in \{0,\dots,N\}.$  By the same argument used to obtain \eqref{10}, we see that the quantity $|y_u(t) - y_v(t)|$ is bounded above by
\begin{equation}
\begin{split}\label{11}
& \sum_{0\leq i\leq i_0} [~~ |w_i| \cdot |f(\alpha f(\alpha f(\beta u(t-2N-i-2)) + \beta u(t-N-i-1))) \\
&\hspace{1cm} - f(\alpha f(\alpha f(\beta v(t-2N-i-2)) + \beta v(t-N-i-1)))| ~~] \\
&+ \sum_{i_0 < i \leq N} |w_i| \cdot |f(\alpha f(\beta u(t-N-i-1)) + \beta u(t-i)) - f(\alpha f(\beta v(t-N-i-1)) + \beta v(t-i))|.
\end{split}
\end{equation}
By the same argument as in \eqref{10}, the first summand in \eqref{11} has the property
\begin{equation}
\begin{split}
& |f(\alpha f(\alpha f(\beta u(t-2N-i-2)) + \beta u(t-N-i-1))) \\
&\hspace{1cm} - f(\alpha f(\alpha f(\beta v(t-2N-i-2)) + \beta v(t-N-i-1)))| \\
&\leq \alpha L\left[ \alpha L^2 \beta |u(t-2N-i-2)-v(t-2N-i-2)| + L\beta|u(t-N-i-1)-v(t-N-i-1)| \right].
\end{split}
\end{equation}
The second summand of \eqref{11} is simply the first summand of \eqref{10}, so similarly
\begin{equation}
\begin{split}
& |f(\alpha f(\beta u(t-N-i-1)) + \beta u(t-i)) - f(\alpha f(\beta v(t-N-i-1)) + \beta v(t-i))| \\
&\leq \alpha L^2 \beta |u(t-N-i-1)-v(t-N-i-1)| + L\beta|u(t-i)-v(t-i)|,
\end{split}
\end{equation}
so that the quantity \eqref{11} is bounded above by
\begin{equation}
\begin{split}\label{12}
& \alpha^2 L^3 \beta\sum_{0\leq i\leq i_0} |w_i| \cdot |u(t-2N-i-2)-v(t-2N-i-2)| \\
&+ \alpha L^2 \beta \sum_{i=0}^N |w_i| \cdot |u(t-N-i-1)-v(t-N-i-1)| + L\beta\sum_{i_0 < i \leq N} |w_i| \cdot |u(t-i)-v(t-i)|.
\end{split}
\end{equation}
Let $C_1 = L\beta.$  For each pair $i,j$ write $z_j^i(t) = t-i-j(N+1)$. Then the square of \eqref{12} becomes
\begin{equation}
\begin{split}\label{13}
& C_1^2 [~~ \alpha^2 L^2 \sum_{0\leq i\leq i_0} |w_i| \cdot |u(z_2^i(t))-v(z_2^i(t))| + \alpha L \sum_i |w_i| \cdot  |u(z_1^i(t))-v(z_1^i(t))| \\
&\hspace{1.5cm} + \sum_{i_0< i\leq N} |w_i| \cdot |u(z_0^i(t))-v(z_0^i(t))| ~~]^2 \\
&\leq 2 C_1^2 [~~ 2\left[\alpha^2 L^2 \sum_{0\leq i\leq i_0} |w_i| \cdot |u(z_2^i(t))-v(z_2^i(t))| \right]^2 + 2 \left[\alpha L \sum_i |w_i| \cdot |u(z_1^i(t))-v(z_1^i(t))| \right]^2 \\
&\hspace{1.5cm} + \left[ \sum_{i_0< i\leq N} |w_i| \cdot |u(z_0^i(t))-v(z_0^i(t))| \right]^2 ~~] \\
&= 2 C_1^2 [~~ 2 (\alpha L)^4 \left[ \sum_{0\leq i\leq i_0} |w_i| \cdot |u(z_2^i(t))-v(z_2^i(t))| \right]^2 + 2(\alpha L)^2 \left[ \sum_i |w_i| \cdot |u(z_1^i(t))-v(z_1^i(t))| \right]^2 \\
&\hspace{1.5cm} + \left[ \sum_{i_0< i\leq N} |w_i| \cdot |u(z_0^i(t))-v(z_0^i(t))| \right]^2 ~~] \\
&\leq 2 C_1^2 (N+1)^2 [~~ 2(\alpha L)^4 \sum_i |w_i|^2 \cdot |u(z_2^i(t))-v(z_2^i(t))|^2 + 2(\alpha L)^2 \sum_i |w_i|^2 \cdot |u(z_1^i(t))-v(z_1^i(t))|^2 \\
&\hspace{2.5cm} + \sum_i |w_i|^2 \cdot |u(z_0^i(t))-v(z_0^i(t))|^2 ~~]
\end{split}
\end{equation}
By \eqref{13}
\begin{equation*}
\|y_u - y_v\|_{\ell^2(\mathbb{R}^{N+2})}^2 \leq (L\beta )^2 (N+1) \left[ 2^2(\alpha L)^4 + 2^2(\alpha L)^2 + 2 \right] |w|^2 \|u-v\|_{\ell^2(\mathbb{R}^M)}^2.
\end{equation*}

\item
In general, on the interval $I_j = [j(N+1),(j+1)(N+1)]$ for $j\in \mathbb{Z}_{\geq 0},$ we can iterate this construction to see that for $t\in I_j,$

\begin{displaymath}
x_i(t) = \left\{
\begin{array}{lr}
f \left( (\alpha f)^{j-1} \left[ \alpha f(\beta u(z_j^i(t))+\beta u(z_{j-1}^i) \right] \right) & ,~ t\geq jN+i+2 \\
f \left( (\alpha f)^{j-2} \left[ \alpha f(\beta u(z_{j-1}^i(t))+\beta u(z_{j-2}^i) \right] \right) & ,~ t < jN+i+2
\end{array}
\right.
\end{displaymath}
where $(\alpha f)^j$ is the $j$th iterate of the function $(\alpha f)(x)=\alpha f(x).$ Note that when the convexity property is used iteratively, it follows the pattern
\begin{equation*}
\left(\sum_{n=1}^N a_n + a_0 \right)^2 \leq 2^N a_1^2 + 2^N a_2^2 + 2^{N-1} a_3^2 + 2^{N-2} a_4^2 + \cdots + 2^3 a_2^2 + 2^2 a_1^2 + 2a_0^2,
\end{equation*}
so the highest power of $2$ is used twice, while the other powers of $2$ descend to $1.$

Iteration of the above argument using the Schwarz inequality and convexity of $x\mapsto x^2$ shows that for $t\in I_j = [j(N+1),(j+1)(N+1)],$
\begin{equation*}
\begin{split}
\|y_u - y_v\|_{\ell^2(\mathbb{R}^{N+2})}^2
&\leq (L\beta)^2(N+1) |w|^2 [ 2 + 2^2(\alpha L)^2 + 2^3(\alpha L)^4 + \cdots \\
&+ 2^{j-1}(\alpha L)^{2(j-2)} + 2^j(\alpha L)^{2(j-1)} + 2^j(\alpha L)^{2j} ] \|u-v\|^2 \\
&\leq (L\beta)^2(N+1) |w|^2 \left[ 1 + \sum_{k=1}^{\infty} 2^k (\alpha L)^{2(k-1)} \right] \|u-v\|^2 \\
&= (L\beta)^2(N+1) |w|^2 \left(1+\frac{2}{1-2\alpha^2 L^2} \right)\|u-v\|_{\ell^2(\mathbb{R}^M)}^2.
\end{split}
\end{equation*}

Let $j'= \min \{j : M\leq (j+1)(N+1)\}.$  Finally, the desired result is attainable since
\begin{equation*}
\begin{split}
\|y_u - y_v\|_{\ell^2(\mathbb{R}^M)}^2 &= \sum_{t=0}^M |y_u(t) - y_v(t)|^2 \\
&\leq \sum_{j=0}^{j'} \sum_{t\in I_j} |y_u(t) - y_v(t)|^2 \\
&= \sum_{j=0}^{j'} \|y_u - y_v\|_{\ell^2(\mathbb{R}^{N+2})}^2~~~(\text{where}~y_u,y_v:I_j \rightarrow \mathbb{R}) \\
&\leq \sum_{j=0}^{j'} (L\beta)^2 (N+1) |w|^2 \left(1+\frac{2}{1-2\alpha^2 L^2} \right)\|u-v\|_{\ell^2(\mathbb{R}^M)}^2 \\
&\leq M (L\beta)^2 (N+1) |w|^2 \left(1+\frac{2}{1-2\alpha^2 L^2} \right)\|u-v\|_{\ell^2(\mathbb{R}^M)}^2. \qedhere
\end{split} 
\end{equation*} 
\end{enumerate}
\end{proof}

While aesthetically pleasing, the upper bound given by Theorem \ref{main_result} is quite crude.  Moreover, this upper bound considers the entire time interval $[0,M],$ but gives no comparison of the two instantaneous reservoir states at time $t\in [0,M]$ with respect to the inputs $u,v.$  It would be ideal to have an upper bound that applies at any time $t.$
\begin{question}
Is there a constant $C=C(N,\alpha,\beta)$ such that $|y_u(t)-y_v(t)| \leq C |u(t)-v(t)|$ for all $t$?
\end{question}

\section{Separation and classification}\label{classification_section}
There are several methods used to determine reservoir quality of ESNs and Liquid State Machines.  In \cite{gibbons} Gibbons provided an overview of the Separation method (see \cite{goodman},\cite{norton}) in great generality.  Our goal is to reduce that overview to the case of the single non-linear node with delay line.  A fundamental difference between Separation for ESNs and Separation for the delay line structure is that with an ESN, there are multiple input nodes.  So if $k$ is the number of input nodes of an ESN, then at a given time $t$ an entire sequence of $k$ data points can be fed into the reservoir simultaneously to yield an instantaneous reservoir state $x(t)$.  For the delay line structure, however, a single input node does not provide such latitude.  If one wishes to feed in a sequence of data (i.e. an input) of length $k$, then $k$ time steps are required instead of just 1.

Let $\mathcal{U}$ be a set of inputs as in Subsection \ref{notation_subsection}.  We assume that at time $t,$ the set of all reservoir states $\mathcal{X}(t)=\{x^{(u)}(t) : u\in\mathcal{U}\}$ is a disjoint union of $N(\mathcal{U})$ prescribed classes $X_1(t), X_2(t),\dots,X_{N(\mathcal{U})}(t)$.  It is desirable for reservoir states within the same class to be close to each other, and for states of different classes to be far from each other.  To that end, it is useful employ several averaging processes that provide insight into the Separation within and between classes.

Several technical definitions are required to introduce the concept of Separation, all of which have been adapted from those in \cite{gibbons}.

\begin{definition} For each $n,$ the \textit{class average (center of mass)} of $X_n(t)$ is the average of all the members of $X_n(t)$:
\begin{equation}\label{class_average}
\mu(X_n(t))= \frac{1}{|X_n(t)|} \sum_{x^{(u)}(t)\in X_n(t)} x^{(u)}(t).
\end{equation}
The \textit{inter-class distance} of $\mathcal{X}(t)$ is the average of all distances between centers of mass in $\mathcal{X}(t)$:
\begin{equation}\label{interclass_distance}
C_d(t)= \frac{1}{N(\mathcal{U})^2}\sum_{n=1}^{N(\mathcal{U})}\sum_{m=1}^{N(\mathcal{U})} |\mu(X_n(t)) - \mu(X_m(t))|.
\end{equation}
\end{definition}

We need a measure of spread within classes in order to define Separation; it is obtained in the following way.  For each class $X_n(t)$, find the average distance from the elements of $X_n(t)$ to the class average $\mu(X_n(t)).$  Then, compute the average of this quantity over all classes.

\begin{definition}\label{intraclass_variance} The \textit{intra-class variance} of $\mathcal{X}(t)$ is
\begin{equation*}
C_v(t)= \frac{1}{N(\mathcal{U})} \sum_{n=1}^{N(\mathcal{U})} \frac{1}{|X_n(t)|} \sum_{x^{(u)}(t)\in X_n(t)} |\mu(X_n(t))-x^{(u)}(t)|.
\end{equation*}
\end{definition}

To perform good classification, the class averages should be separated by an amount substantial enough to differentiate them.  Moreover, the members of a given class should not stray too far from their class average, or else they could be misclassified.  Therefore the following definition provides a tool for measuring the effectiveness of a single-node reservoir classification system.
\begin{definition}\label{separation}
The \textit{separation of $\mathcal{X}(t)$} is
\begin{equation}
\text{Sep}_{\mathcal{X}}(t)=\frac{C_d(t)}{C_v(t)+1}.
\end{equation}
\end{definition}

Definition \ref{separation} first appeared in \cite{goodman}, in which Goodman obtained results yielding large positive correlations between separation and classification accuracy of the output function.  This was done for temporal pattern classification problems in the context of randomly generated reservoirs~\cites{gibbons,goodman}.  So at time step $t$, if $C_d(t)$ is large while  $C_v(t)$ is relatively small, one expects high classification accuracy.  It is shown in \cite{gibbons} that a hallmark of good classification for random reservoirs is the existence of a constant $C>0$ such that
\begin{equation}\label{inverse_lip}
|x^{(u)}(t)-x^{(v)}(t)| \geq C|u(t)-v(t)|
\end{equation}
for all pairs $u,v\in\mathcal{U}$ such that $|u(t)-v(t)|$ is ``large."  For sake of simplicity, let us say that \eqref{inverse_lip} should hold whenever $|u(t)-v(t)|\geq 1.$

\subsection{Injectivity and periodicity of $f$}
If the chosen nonlinear function $f$ is injective (which is fairly typical considering the common usage of $f=\tanh$), then an inequality like \eqref{inverse_lip} is not outside the realm of possibility, since the function $u(t)\mapsto x^{(u)}(t)$ is also injective.

\begin{fact}
Let $X(N,f,\beta,\alpha)$ be a reservoir where $f$ is injective. If two inputs $u$ and $v$ are such that $u(t)\neq v(t),$ then $x^{(u)}(t)\neq x^{(v)}(t).$
\end{fact}
\begin{proof}
Suppose $u(t_0)\neq v(t_0)$ and assume $x^{(u)}(t_0)=x^{(v)}(t_0).$  Put $t_1=\min\{t:u(t)\neq v(t)\}$.  By definition \ref{resdef},
\begin{equation}
x_1^{(u)}(t_1) = f(\alpha x_N^{(u)}(t_1 -1)+\beta u(t_1)),\\
\end{equation}
and since $u(t_1-1)=v(t_1-1),$ we have
\begin{equation}
\begin{split}\label{22}
x_1^{(v)}(t_1) &= f(\alpha x_N^{(v)}(t_1 -1)+\beta v(t_1)) \\
&= f(\alpha x_N^{(u)}(t_1 -1)+\beta v(t_1)).
\end{split}
\end{equation}
Then $x_1^{(u)}(t_1)=x_1^{(v)}(t_1)$ since $x^{(u)}(t_0)=x^{(v)}(t_0),$ so it follows from \eqref{22} that
\begin{equation}\label{23}
\alpha x_N^{(u)}(t_1 -1)+\beta u(t_1) = \alpha x_N^{(u)}(t_1 -1)+\beta v(t_1).
\end{equation}
By \eqref{23} it is clear that $u(t_1)=v(t_1)$ since $\beta >0,$ which is a contradiction.
\end{proof}

Suppose that we graph $u$ as a function of time $t,$ and that we desire a single-node reservoir $X(N,f,\beta,\alpha)$ whose state is not extremely sensitive to vertical translations.  Intuitively, it seems reasonable to use a nonlinear function $f$ that is periodic.  The following fact shows that periodic functions $f$ can, in some sense, yield periodic reservoir states.

\begin{fact}\label{periodic_fact}
Let $X(N,f,\beta,\alpha)$ be a reservoir with inputs $u$ and $v,$ where $f$ is $P$-periodic.  For all $t$, if $v(t)=u(t)-\nicefrac{P}{\beta}$, then $x^{(u)}(t)=x^{(v)}(t).$
\end{fact}
\begin{proof}
The following basic argument proceeds by induction on $t,$ but this statement could be verified directly from the equations for $x_i(t)$ given in Part IV of the proof of Theorem \ref{main_result}.

The case $t=0$ is trivial.  For the base case, let $t=1$ and note that
\begin{equation*}
x_1^{(u)}(t)=f(\alpha x_N^{(u)}(0) + \beta u(1)) = f(\alpha x_N^{(v)}(0) + \beta v(1)+P)=x_1^{(v)}(1).
\end{equation*}
Clearly $x_i^{(u)}(1)=0=x_i^{(v)}(1)$ for all $1< j \leq N,$ so that $x^{(u)}(1)=x^{(v)}(1).$  Now assume $x^{(u)}(k)=x^{(v)}(k).$  Then
\begin{equation}\label{24}
x_1^{(u)}(k+1)=f(\alpha x_N^{(u)}(k) + \beta u(k+1)) = f(\alpha x_N^{(v)}(k) + \beta v(k+1) + P) = x_1^{(v)}(k+1),
\end{equation}
so it remains to show that $x_i^{(u)}(k+1) = x_i^{(v)}(k+1)$ for $1<i\leq N.$  Since $x^{(u)}(k)=x^{(v)}(k),$
\begin{equation*}
\begin{split}
x_2^{(u)}(k+1) &= x_1^{(u)}(k) = x_1^{(v)}(k) = x_2^{(v)}(k+1) \\
x_3^{(u)}(k+1) &= x_2^{(u)}(k) = x_2^{(v)}(k) = x_3^{(v)}(k+1) \\
&\vdots \\
x_N^{(u)}(k+1) &= x_{N-1}^{(u)}(k) = x_{N-1}^{(v)}(k) = x_N^{(v)}(k+1). \qedhere
\end{split}
\end{equation*}
\end{proof}

\begin{remark}
In the case $f=\sin$, Fact \ref{periodic_fact} demonstrates the importance of pre-conditioning reservoir inputs $u$ so that $u(t)\in [-1,1]$ for all $t\in\mathbb{Z}_{\geq 0}$ and for all $u\in\mathcal{U}$.  Depending on the task, it may or may not be desirable to put $(u(1),\dots,u(M))$ and $v=(u(1)+ \nicefrac{P}{\beta}, u(2)+\nicefrac{P}{\beta},\dots,u(M)+\nicefrac{P}{\beta})$ into the same class, since the period $P$ can be made arbitrarily large.  So choosing $f$ to be periodic can yield the undesirable pair of separation properties $|u(t)-v(t)|=\nicefrac{P}{\beta}$ and $|x^{(u)}(t) -x^{(v)}(t)|=0,$ which is a blatant violation of condition \eqref{inverse_lip}.  Since $f=\sin$ is injective on $[-\nicefrac{\pi}{2},\nicefrac{\pi}{2}],$ such difficulties can easily be avoided by normalizing each $u\in\mathcal{U}$, which gives $|u(t)|\leq \|u\|_2\leq 1.$
\end{remark}

It is not clear whether any or all of the aforementioned quality metrics are useful for the single-node with delay line reservoir model.  Since the reservoir states rapidly change over time while receiving a signal, it is quite difficult to produce a lower bound on the difference in reservoir states in terms of the difference in inputs such as \eqref{inverse_lip}.


\subsection*{Acknowledgement}
This article is based on research performed during the summer of 2015 by the author at Air Force Research Laboratory in Rome, NY under the direction of Ashley Prater.

\end{document}